\documentclass [article] {amsart}
\usepackage{amsfonts,amsthm,amsbsy,amsmath,amssymb}
\usepackage{latexsym,enumerate,enumitem}
\usepackage{tikz-cd}
\usetikzlibrary{positioning}
\usepackage{xcolor}
\usepackage{mathrsfs}

\newtheorem{corollary}{Corollary}

\newtheorem{prop}{Proposition}
\newtheorem{theorem}{Theorem}
\newtheorem{remark}{Remark} 
\newtheorem{definition}{Definition}

\newtheorem*{theorem*}{Theorem}
\newtheorem*{lemma*}{Lemma}
\newtheorem*{prop*}{Proposition}
\newtheorem*{corollary*}{Corollary}
\newtheorem*{remark*}{Remark} 
\newtheorem*{remarks*}{Remarks}

\def\f{\frac}

\newcommand{\esssup}{\mathop{\textup{ess.sup}}}
\newcommand{\vp}{\varphi}

\def\C{{\mathbb C}}

\def\R{{\mathbb R}}

\def\a{\alpha}

\def\s{\sigma}

\DeclareFontFamily{U}{mathx}{\hyphenchar\font45}
\DeclareFontShape{U}{mathx}{m}{n}{
	<5> <6> <7> <8> <9> <10>
	<10.95> <12> <14.4> <17.28> <20.74> <24.88>
	mathx10
}{}
\DeclareSymbolFont{mathx}{U}{mathx}{m}{n}
\DeclareFontSubstitution{U}{mathx}{m}{n}
\DeclareMathAccent{\widecheck}{0}{mathx}{"71}
\DeclareMathAccent{\wideparen}{0}{mathx}{"75}

\begin{document} 
\title[On Families of maximal operators]{On Families between the Hardy-Littlewood and Spherical maximal functions}
\author[Georgios Dosidis and Loukas Grafakos]{Georgios Dosidis and Loukas Grafakos}

\newcommand{\Addresses}{{
		\bigskip
		\footnotesize
		
		\textsc{Georgios Dosidis.}
		\textsc{Department of Mathematics, University of Missouri,
			Columbia MO 65203}\par\nopagebreak
		\textit{E-mail address:} \texttt{ganq8f@mail.missouri.edu}
		
			\medskip
		\textsc{Loukas Grafakos.}
		\textsc{Department of Mathematics, University of Missouri,
			Columbia MO 65203}\par\nopagebreak
		\textit{E-mail address:} \texttt{grafakosl@missouri.edu}

}}
\thanks{The authors acknowledge the support of the Simons Foundation.}

\begin{abstract} We study a family of   maximal operators that provides a continuous link   connecting the   Hardy-Littlewood maximal function to the   spherical maximal function. Our theorems  are   proved
 in  the multilinear setting  but may contain 
new  results even in the linear case. For this family of operators we obtain  bounds between 
Lebesgue spaces in the optimal range of exponents. 
\end{abstract}
\maketitle

\section*{Introduction}
Spherical averages arise naturally in PDE but       $L^p$ bounds for maximal spherical averages
 were first obtained by Stein~\cite{S1976}, who showed that 
 the spherical maximal function 
\begin{equation}\label{SS}
  S (f) (x) := \sup_{t>0} \frac{1}{\omega_{n-1}}\int_{\mathbb{S}^{n-1}}|f(x-t \theta)| \, d\s_{n-1}(\theta)
\end{equation}
is bounded from $L^p(\R^n) $ to $L^p(\R^n) $ when    $p>\frac{n}{n-1}$ and $n\geq 3$ and     is unbounded  when $p\leq\frac{n}{n-1}$ and $n\geq 2$. The positive direction of this result was later extended to the case $n=2$ by Bourgain~\cite{B1986}. Here $d\sigma_{n-1}$ is the canonical surface measure on the sphere and $\omega_{n-1} $ is the measure of the entire unit sphere.  A number of other authors have also studied the spherical maximal function; see for instance \cite{C1985}, \cite{CM1979}, \cite{MSS1992}, and \cite{S1998}. Extensions of the spherical maximal function to different settings have also been   
considered; see \cite{C1979}, \cite{G1981}, \cite{DV1996}, and \cite{MSW2002}.  

The boundedness of  the maximal operator $  S$    in ~\cite{S1976} was obtain via the auxiliary  family of operators
\begin{equation}\label{Sa} 
S_\a (f)(x) = \sup_{t>0} \frac{2}{\omega_{n-1} B(\f n2,1-\a)} \int_{\mathbb{B}^{n}} |f(x-ty)| (1-|y|^2)^{-\a } dy ,
\end{equation}
defined originally for Schwartz functions, where $0\le  \a< 1$. Here $\mathbb{B}^{n}$ is the unit ball in $\R^n$,   $B$ is the beta function defined by $B(x,y) = \int_0^1t^{x-1} (1-t)^{y-1} dt$ for $x,y>0$. For each $0<\a<1$, Stein obtained boundedness for the operator $S_\a$ from $L^p$ to itself in the optimal range of exponents: $p>\frac{n}{n- \a}$, when $n\ge 3$.  This was extended to the case $n=2$ indirectly in~\cite{B1986} and more explicitly  in~\cite{MSS1992}.   In \cite{IS2002} the authors obtained boundedness results for maximal operators associated to a more general set of measures that includes the  family $S_\a$. 

Recall another classical averaging operator, the Hardy-Littlewood maximal  function 
\begin{equation*} 
M (f)(x) = \sup_{t>0} \frac{1}{v_n}\int_{\mathbb{B}^{n}} |f(x-ty)| dy.
\end{equation*}
Here $f$ is a locally integrable function on $\R^n$ and $v_n$ is the volume of $\mathbb{B}^n$.

The relationship between the aforementioned operators is as follows: The family $S_\a$ provides a continuous link that connects $M$ to $S$ in the  following explicit way:  
For any $f\in {L}_{loc}^1(\R^n) $ and any $x\in\R^n$ we have
\begin{align*}   & M(f) (x)\leq S_\a (f)(x) \leq S(f)(x), \\
& \lim\limits_{\a\to 1^-} S_\a (f)(x) = S(f)(x), \\
 &  \lim\limits_{\a\to 0^+}  S_\a (f)(x) = M(f)(x).
\end{align*}
These assertions are contained in Theorem~\ref{t1} and are proved in the next section.

In this paper, we denote by $d\sigma_{\kappa-1}$   the surface measure on unit sphere  $\mathbb{S}^{\kappa-1}$ in $\mathbb R^\kappa$, $v_{\kappa}$   the measure of the unit ball  in $\R^{\kappa}$ and $\omega_{\kappa-1} = d\sigma_{\kappa-1} ( \mathbb{S}^{\kappa-1})$ is  the total  
measure of $\mathbb{S}^{\kappa-1}$. Recall that 
$\kappa v_\kappa=\omega_{\kappa-1}$ for any   integer $\kappa\ge 2$.  We also use the notation 
$\mathbb B^\kappa$ for the unit ball in $\mathbb R^\kappa $ and $R\,\mathbb B^\kappa$ for the ball of radius $R>0$ 
centered at the origin in $\mathbb R^\kappa $. 
The space of Schwartz functions on $\mathbb R^\kappa $ is denoted by 
$\mathcal S( \mathbb R^\kappa )$. 

Our purpose is to study multilinear versions of $S$, $S_\a$ and of $M$. 
We define a  multi(sub)linear  maximal operator  as follows: 
\begin{equation}\label{MMa} 
M^m (f_1,\dots,f_m)(x) = \sup_{t>0} \frac{1}{v_{mn } }\int_{\mathbb{B}^{mn}}
 \prod_{i=1}^m |f_i(x-ty_i)| \, dy_1\cdots dy_m. 
\end{equation}
The uncentered version of this maximal operator first appeared in the work of Lerner,  Ombrosi, Perez, Torres, Trujillo-Gonzalez~\cite{LOPTT2009}  with the unit cube in place of the unit ball.
 Next, we introduce the    family of operators
\begin{equation}\label{MSa} 
S^m_\a (f_1,\dots,f_m)(x) =  
 \frac{2/ \omega_{ mn-1}  }{B(\f{mn}2,1-\a)} \sup_{t>0}  \int_{\mathbb{B}^{mn}} \prod_{i=1}^m |f_i(x-ty_i)|  \,
  \frac{dy}{ (1-|y|^2)^{ \a } } ,
\end{equation}
defined initially for functions $ f_i  \in  \mathcal{S}(\R^n) $ and $0\le \a<1$.   This is a multilinear extension 
of the operator $S_\a = S_\a^1$ introduced in \eqref{Sa}. 

We recall the definition of the  multilinear spherical maximal operator  
\begin{equation}\label{SSa} 
S^m  (f_1,\dots,f_m)(x) = \sup_{t>0} \frac{1}{\omega_{mn-1 } }\int_{\mathbb{S}^{mn-1}} \prod_{i=1}^m |f_i(x-t\theta_i)|  \,  d\sigma_{mn-1}(\theta_1,\dots , \theta_m), 
\end{equation}
given also for   functions $ f_i  \in  \mathcal{S}(\R^n) $. When $m=1$, $S^m$ reduces to $S$ in \eqref{SS}. 
The bilinear analogue of Stein's spherical maximal function (when m=2) was first introduced in \cite{GGIPS2013}  by Geba, Greenleaf, Iosevich, Palsson, and Sawyer  who obtained the first bounds for it but later improved bounds were provided  by \cite{BGHHO2018}, \cite{GHH2018}, \cite{HHY2019}, and \cite{JL2019}. A multilinear (non-maximal) version of this operator when all input functions lie in the same space $L^{p}(\R)$ was previously studied by Oberlin \cite{O1988}. The authors in \cite{BGHHO2018} provided an example that shows that the bilinear spherical maximal function is not bounded when $p\geq \frac{n}{2n-1}$. Last year Jeong and Lee in \cite{JL2019} proved that the bilinear maximal function is pointwise bounded by the product of the linear spherical maximal function and the Hardy-Littlewood maximal function, which helped them establish boundedness in the optimal open set of exponents, along with some endpoint estimates. These results were extended to the multilinear setting in~\cite{D2019}. Recently certain analogous bounds have been obtained by Anderson and Palsson  in \cite{AP20191},   \cite{AP20192} concerning a discrete version of the multilinear spherical maximal function.

We would like to extend the definitions of the    operators in \eqref{MSa} and \eqref{SSa} 
 to functions in $f_i $ in $ L^1 _{loc} (\R^n) $. Fix $f_i $ in $ L^1 _{loc} (\R^n) $ and $x\in \R^n$;    then    
\begin{equation}\label{SSa2} 
t\mapsto F(t)= t^{mn-1} \int_{\mathbb{S}^{mn-1}} \prod_{i=1}^m |f_i(x-t\theta_i)|  
d\sigma_{mn-1}(\theta_1,\dots , \theta_m)
\end{equation}
is integrable over any   interval $[0,L]$, which implies that the   integrals    
in \eqref{SSa} are finite for almost all $t>0$.  
Likewise, if $F$ is as in \eqref{SSa2} and $t\in (0,L)$, then 
\begin{equation}\label{MSa2} 
  \int_{\mathbb{B}^{mn}} \prod_{i=1}^m |f_i(x-ty_i)|   \frac{dy}{ (1-|y|^2)^{ \a } } = \int_0^1 \f{ F(tr)}{ (1-r^2)^{ \a} }\f{dr}{  t^{mn-1}  } \le \f{1}{  t^{mn-\a}  } \int_0^L \f{ F(s)\,ds}{ (t-s)^{ \a} },
\end{equation} 
and the last integral is the convolution (evaluated at $t$) 
of the $L^1$ functions $F\chi_{[0,L]} $ and $s^{-\a} \chi_{(0,L]}$ on the real line, hence it is finite  
a.e. on $  (0,L)$. We conclude that 
the integral in \eqref{MSa} is finite for almost all $t>0$  for $f_i\in L^1_{loc}(\R^n)$ and $x\in \R^n$.

Now, one cannot properly define the supremum of a family $\{A_t\}_{t>0}$ ($A_t\ge 0$) which satisfies $A_t<\infty$ 
 for almost all $t>0$. But it 
is possible to define   the essential supremum of $\{A_t\}_{t>0}$, which  is  practically the supremum restricted over the subset 
of $(0,\infty)$  { of full measure}. So to extend the definitions of the operators in 
\eqref{SSa} and \eqref{MSa} to functions $f_i\in L^1_{loc} (\R^n)$ for any $x\in \R^n$ by     replacing the supremum in these 
expressions by the essential supremum $\esssup$.  However, this adjustment is not needed when $f_i\in L^{p_i}(\R^n)$ with $\sum_{i=1}^n \frac{1}{p_i} = \frac{1}{p} < \frac{mn-\a}{n}$, since, in that case, the corresponding averages vary continuously in $t$. See Corollary~\ref{C1} below. 
Based on this discussion we provide the following definition.

\begin{definition} \label{Def1}
Let  $t>0$,  $f_i\in L^1_{loc}(\R^n)$ for $1\le i\le m$, and $x\in \R^n$.  We define 
\begin{equation}\label{Smat}
S^m_{\a,t}(f_1,\dots,f_m)(x) = \frac{  {2}/{\omega_{mn-1}}}{B(mn/2,1-\a)} \int_{\mathbb{B}^{mn}} \prod_{i=1}^m  f_i(x-ty_i) \,   \frac{dy}{ (1-|y|^2)^{ \a } }  
\end{equation}
and
\begin{equation}\label{ess1}
S^m_\a (f_1,\dots,f_m)(x) =   \esssup_{t>0}  S^m_{\a,t}(| f_1|,\dots, |f_m|)(x)
\end{equation}
for $0\le \a<1$. We also define 
\begin{equation}\label{Smat2}
S^m_{1,t}  (f_1,\dots,f_m)(x) =  \frac{1}{\omega_{mn-1 } }\int_{\mathbb{S}^{mn-1}} 
\prod_{i=1}^m  f_i(x-t\theta_i)   \,   d\sigma_{mn-1}(\theta_1,\dots , \theta_m)   
\end{equation}
and 
\begin{equation}\label{ess2}
S^m  (f_1,\dots,f_m)(x) = \esssup_{t>0} S^m_{1,t}  (|f_1|,\dots,|f_m|)(x) . 
\end{equation}
\end{definition}

In this paper we prove the following results:

\begin{theorem}\label{t1} 
Let $0<\a<1$. Given   $f_i\in  L^1 _{loc} (\R^n) $   and   $x\in\R^n$ we have
\begin{align} 
 & M^m (f_1,\dots,f_m)(x)\leq S^m_\a (f_1,\dots,f_m)(x)\leq S^m  (f_1,\dots,f_m)(x)  \label{lim}\\
& \lim\limits_{\a\to 1^-} S^m_\a (f_1,\dots,f_m)(x) = S^m  (f_1,\dots,f_m)(x) . \label{lim2}
 \end{align}
  { These statements are valid even when some of the preceding  expressions 
   equal   $\infty$. Moreover, if if $ S^m_{\a_0} (f_1,\dots,f_m)(x)<\infty$ 
for some $\alpha_0\in (0,1)$ then 
   \begin{equation} 
    \lim\limits_{\a\to 0^+}  S^m_\a (f_1,\dots,f_m)(x)= M^m  (f_1,\dots,f_m)(x) . \label{lim3} 
 \end{equation} 
 }  
\end{theorem}

\begin{remark} To see that condition $ S^m_{\a_0} (f_1,\dots,f_m)(x)<\infty$ 
	for some $\alpha_0\in (0,1)$ is necessary in order for \eqref{lim3} to hold, consider the function 
	\[  
	f(y) = \f{1}{  1-|y|^2  } \f{ 1}{\big( \log \f{1}{  1-|y|^2}\big)^2  }  \chi_{|y|<1} .
	\]
For this function the property $\lim_{\alpha\to 0} S_\alpha f(y) = Mf(y)$ fails at $y=0$, since for all $a>0$ we have that $S_\a f(0) =\infty$ while $Mf(0)<\infty$. 
\end{remark}

As $M^m$ is pointwise controlled by the product of the Hardy-Littlewood operators acting on each function, this 
operator is bounded from $L^{p_1}(\R^n)\!\times \dots \times\! L^{p_m}(\R^n)$ to $L^p(\R^n)$ in the full 
range of exponents $1<p_1,\dots , p_m\le \infty$ and $1/p_1+\cdots +1/p_m=1/p$. 
Boundedness for $S^m$ holds in the smaller region $ {n}/{(mn- 1)}<p\leq \infty$ as shown in~\cite{D2019}. So it is 
expected that $S^m_\a$ are bounded in some intermediate regions. This is the content of the following result.  

\begin{theorem}\label{t2} Let $n\geq 2$, $0\le \a<1$, and $1< p_i\leq \infty$. Define $p$  by $\sum_{i=1}^m \frac{1}{p_i} = \frac{1}{p}$. Then there is a constant $ C =C(m,\a,p_1,\dots,p_m)$ such that
\begin{equation}\label{B} 
\|S^m_\a (f_1,\dots,f_m) \|_{L^p(\R^n)}\leq C \prod_{i=1}^m \|f_i\|_{L^{p _i}(\R^n)} 
\end{equation}
for all  $f_i\in L^{p_i} (\R^n)$  if and only if  
$$
\frac{n}{mn- \a}<p\leq \infty. 
$$ 
Moreover, if \eqref{B} holds, then the constant $C$ can be chosen to be independent of the dimension 
(as indicated by the parameters on which it is claimed to depend).   
\end{theorem}

 We graph the range of boundedness for the bilinear operator $S_\a^2$.

\newcommand{\xx}{3.5}
\newcommand{\yy}{3.5}
\begin{figure}[h]
	\begin{center}
		\begin{tikzpicture}
			\coordinate (O) at (0,0);
			\coordinate (A) at (\xx,0);
			\coordinate (B) at (0,\yy);
			\coordinate (C) at (\xx,\yy);
			\coordinate (s1) at (\xx,0.62*\yy);
			\coordinate (s2) at (0.62*\xx,\yy);
			\coordinate (a1) at (\xx,0.83*\yy);
			\coordinate (a2) at (0.83*\xx,\yy);
			\coordinate (X) at (1.28*\xx,0);
			\coordinate (Y) at (0,1.28*\yy);

			\draw[white,fill=yellow!16] (O) -- (A) -- (a1) -- (a2) -- (B) -- (O) -- cycle;
			\draw[white,fill=yellow!02.5] (a1) -- (C) -- (a2) -- (a1) -- cycle;

			\draw[blue] (O) -- (A);
			\draw[blue] (O) -- (B);
			\draw[blue] (A) -- (a1);
			\draw[blue] (B) -- (a2);
			\draw[red!80] (a1) -- (C);
			\draw[red!80] (a2) -- (C);
			\draw[black,densely dotted] (s1) -- (s2);
			\draw[red!80] (a1) -- (a2);

			\foreach \rr in {a1,a2} {\filldraw[red](\rr) circle(1pt);}
			\foreach \rr in {O, A, B} {\filldraw[blue](\rr) circle(1pt);}
			\foreach \rr in {s1,s2} {\filldraw[black](\rr) circle(1pt);}
			\node [below  =1mm of A]  {(1,0)};
			\node [left  =1mm of B]  {(0,1)};
			\node [right =1mm of s1]  {$(\frac{n-1}{n} , 1)$};
			\node [right=1mm of a1]  {$(\frac{n-\a}{n}, 1)$};
			\node [below left = 0.5mm of O]  {(0,0)};

			\node [above  =1mm of X]  {$1/p_1$};
			\node [left =1mm of Y] {$1/p_2$};
			\draw[->,black,densely dotted] (A) -- (X);
			\draw[->,black,densely dotted] (B) -- (Y);

		\end{tikzpicture}
		\caption[Figure 1]{Range of $L^{p_1}\times L^{p_2}\to L^p$ boundedness of $S^2_\a$ when $0\leq\a\leq~1$ and $n\geq 2$. The bilinear spherical maximal function is bounded below the black dotted line, while the bilinear Hardy-Littlewood maximal function is bounded on the entire square. }\label{F1}
	\end{center}
\end{figure}


\begin{remark}\label{r1}
As a consequence, we obtain dimensionless $L^{p_1}\times\cdots\times L^{p_m}\to L^{p}$ bounds for the multilinear maximal function $M^m$ for all $\f1m <p \leq \infty$; this extents the   result of Stein and Str\"omberg~\cite{SS1983} to the multilinear setting. 
\end{remark}

The estimates in \eqref{B} imply that when  $f_i\in L^{p_i}(\R^n)$ with $\sum_{i=1}^n \frac{1}{p_i} = \frac{1}{p} < \frac{mn-\a}{n}$, then  for almost all $x\in\R^n$, $S^m_{\a,t}(f_1,\dots,f_m)(x)$ 
are finite  uniformly in $t>0$. 

\begin{corollary}\label{C1} 
Let  $0\le \a\le 1$ and suppose that for all $1\le i\le m$,  $f_i\in L^{p_i}(\R^n)$ where $1<p_i \le \infty$  satisfy  $\sum_{i=1}^m \frac{1}{p_i} = \frac{1}{p} < \frac{mn-\a}{n}$. Then for almost every $x$ in $\R^n$,  the function $t\mapsto S^m_{\a,t}(f_1,\dots,f_m)(x)$ is well defined and continuous in $t\in(0,\infty)$. Therefore   in Definition~\ref{Def1}, for almost all $x\in \R^n $,  we can replace 
the essential supremum    by  a supremum in both \eqref{ess1} and \eqref{ess2}.
\end{corollary} 


\begin{corollary}\label{C2} 
Let  $0\le \a\le 1$ and suppose that for all $1\le i\le m$,  $f_i\in L^{p_i}_{loc}(\R^n)$ where $1<p_i \le \infty$  satisfy  $\sum_{i=1}^m \frac{1}{p_i} = \frac{1}{p} < \frac{mn-\a}{n}$. Then for almost every $x\in\R^n$, 
\begin{equation}\label{99} 
\lim_{t\to 0} S^m_{\a,t}(f_1,\dots,f_m)(x) = f_1(x)\cdots f_m(x). 
	\end{equation}
\end{corollary}

Parts of  Theorem~\ref{t1} may be new even when $m=1$. Theorem~\ref{t2} is only new when $m\ge 2$ as
the case $m=1$ was considered in~\cite{S1976}.  
The proofs of these theorems can be suitably adapted to the measures 
	$$
	\f{ 	q } { B( {mn}/q ,1-\a)} \f{d\vec y}{(1-|\vec y\, |^q)^{ \a}} 
	$$
	for any $q>0$ in lieu of 	  
	$$
	\f{ 	2 } { B(\f{mn}2,1-\a)} \f{d\vec y}{(1-|\vec y\, |^2)^{ \a}} 
	$$
in \eqref{MSa}. To simplify the notation in our proofs, we adopt the following conventions: 
\begin{center}
\begin{tabular}{  l l l }
& $\vec y = (y_1,\dots , y_m)\in( \R^{n})^m$  &\qquad $ [\vec f\,]  =   (f_1,\dots , f_m) $\\
& $d\vec y  = dy_1\cdots dy_m$        & \qquad    $ ( f_1\otimes \cdots \otimes f_m ) (\vec y\,)= f_1(y_1) \cdots   f_m(y_m)$ \\
& $\vec \theta = (\theta_1,\dots , \theta_m)\in  \mathbb S^{mn-1}$  &
\qquad    $\otimes \vec f  =  f_1\otimes \cdots \otimes f_m $  \\
& $\overline{x}= \underbrace{(x,\dots, x)}_{\textup{$m$ times} } \in( \R^{n})^m$ &
\qquad    $| \!\otimes\! \vec f  \,| = | f_1| \otimes \cdots \otimes |f_m |$
\end{tabular}
 \end{center}

 Two main ideas are used in the proof of 
Theorem ~\ref{t1};  integration by parts  and 
  the fundamental theorem of calculus, both 
  with respect to the radial coordinate. Theorem  ~\ref{t2} is based on 
 a slicing formula that allows us to control $S^m_\alpha$ by the product 
 of the  Hardy-Littlewood maximal operators acting on $m-1$ input functions
 and  of $S_\alpha$ (defined in \eqref{Sa}) acting on the remaining function. 
 This gives estimates near the vertices of the region on which boundedness is claimed, 
 while the remaining bounds are obtained by multilinear interpolation.

\section*{The proof of Theorem~\ref{t1}}
  
  Before we discuss the proof of Theorem~\ref{t1} we note that 
  when $\a=0$,  equality holds in the first inequality in \eqref{lim}, since 
$$
\f{\omega_{mn-1} }2 B\left(\f{mn}2,1\right) =\f{\omega_{mn-1}}2   \frac{2}{mn} = v_{mn} . 
$$
That is, $M^m [\vec f\,](x)= S_0^m [\vec f\,](x)$. 
In fact \eqref{lim} is   valid even when $\a=0$, since 
\begin{align*}
M^m [\vec f\,](x) = S_0^m [\vec f\,](x) &= \sup_{t>0} \frac{1}{v_{mn}} \int_0^1 \int_{\mathbb S^{mn-1}} | \!\otimes\! \vec f  \,|(\overline{x}-tr\vec \theta\, ) d\sigma_{mn-1} 
(\vec \theta\,) r^{mn-1} dr\\
&\leq  \frac{1}{v_{mn}} \int_0^1 \sup_{t'>0} \int_{\mathbb S^{mn-1}} | \!\otimes\! \vec f  \,| (\overline{x}-t'\vec \theta\, ) d\sigma_{mn-1}(\vec \theta\,)  r^{mn-1} dr\\
&= mn\, S^m[\vec f\,] (x) \int_0^1 r^{mn-1} dr\\
&= S^m[\vec f\,](x).
\end{align*}

\begin{proof}[Proof of Theorem \ref{t1}]
First we show that for any $0<\a<1 $  we have $S_\a^m [\vec f\,](x) \leq S^m[\vec f\,](x)$ for any $x\in \R^n$.
Indeed, we have  
\begin{align*}
& \frac{1}{\omega_{mn-1}}\frac{2}{B(\f{mn}2,1-\a)}\esssup\limits_{t>0} \int_{\mathbb{B}^{mn}}
| \!\otimes\! \vec f  \,| (\overline{x}-t \vec y\, ) (1- |\vec y\,|^2)^{-\a } d\vec y\\
\leq& \frac{1}{\omega_{mn-1}}\frac{2}{B(\f{mn}2,1-\a)} \int_0^1 \f{ r^{mn-1} }{ (1-r^2)^{ \a }} \esssup_{t>0} \int_{\mathbb S^{mn-1}}
| \!\otimes\! \vec f  \,|(\overline{x}- r t \vec \theta \, )   d\sigma_{mn-1} (\vec \theta\, ) dr\\
\leq& \frac{1}{\omega_{mn-1}}\frac{2}{B(\f{mn}2,1-\a)} \bigg( \int_0^1 \f{ r^{mn-1} }{ (1-r^2)^{ \a }}dr\bigg)
 \esssup_{t'>0} \int_{\mathbb S^{mn-1}} 
| \!\otimes\! \vec f  \,| (\overline{x} -t' \vec \theta\,) d\s(\vec \theta\,) \\
= &S^m [\vec f\,] (x),
\end{align*}
as the   $r$ integral in the parenthesis is equal to $\f12B(\f{mn}2,1-\a)$. This concludes the proof of the second inequality in \eqref{lim}. 

Next we  prove the first inequality in \eqref{lim}. That is, for a fixed $x\in \R^n$ and $0< \a  < 1$, we show that 
$M^m [\vec f\,]  (x)\leq S_\a^m[\vec f\,] (x)$.   If for some $x\in \R^n$ we had $M^m[\vec f\,](x) =\infty$, 
we would also have that $S_\a^m[\vec f\,] (x)=\infty$ as $(1-|\vec y \, |^2)^{-\a}\ge 1$ when $|\vec y \, |< 1$. 
  So we may assume that $M^m[\vec f\,](x) <\infty$ in the calculation below. 
For fixed $t>0$ we define 
\[
H_t(r) = \int_0^r s^{mn-1} \bigg( \int_{\mathbb S^{mn-1}} 
| \!\otimes\! \vec f  \,|(\overline{x}-ts\vec \theta\,)\, d\s(\vec \theta\,) 
\bigg) ds = \int_{|\vec y\, |\le r}| \!\otimes\! \vec f  \,| (\overline{x}-t \vec y\,) d\vec y, 
\]
for $r>0$. As each $f_j$ is locally integrable, the integral on the right converges absolutely, and thus the 
expressions in the parentheses are finite for almost all $s>0$ and 
moreover, the   $s$-integral  converges absolutely.  Thus $H_t(r) $ is the integral from $0$ to $r$ of an $L^1$ function.
Then,   the Lebesgue differentiation theorem gives 
$$
\f{d}{dr} H_t(r) = H_t'(r) =r^{mn-1}  \int_{\mathbb S^{mn-1}}  | \!\otimes\! \vec f  \,|
(\overline{x}-tr\vec \theta\,) d\s(\vec \theta\, )  \qquad
\textup{for almost all $r>0$.} 
$$
 Moreover,  for any $r>0$ we have 
 $$ 
 \esssup\limits_{t>0} \frac{1}{v_{mn} r^{mn}} H_t(r) =  \esssup\limits_{t>0} \frac{1}{v_{mn}  } H_{rt} (1) = 
  \esssup\limits_{t'>0} \frac{1}{v_{mn}  } H_{t'} ( 1) = M^m [\vec f\,](x) <\infty,
 $$
where in the last equality we replaced the essential supremum by the supremum, 
 using the continuity of the function 
$$
t\mapsto M_t^m(f_1,\dots,f_m)(x) =  \frac{1}{v_{mn } }\int_{\mathbb{B}^{mn}} \prod_{i=1}^m |f_i(x-ty_i)|  dy_1\cdots dy_m, 
$$
 for any $f_i\in L^1_{loc}(\R^n)$, which can be obtained by an application  of the Lebesgue dominated convergence theorem. 
Let 
$$
c_{mn,\a}=  \f{2}{ \omega_{ mn-1}B(\f{mn}2,1-\a) }   . 
$$
For any $0<b<1$ we write 
\begin{align*}
S_\a^m &[\vec f\,]  (x) \\
&\geq \esssup\limits_{t>0} c_{mn,\a}\int_0^b H'_t(r) \f{1}{ (1-r^2)^{ \a } } dr\\
&= \esssup\limits_{t>0} c_{mn,\a} \left[H_t(b) \f{1}{ (1-b^2)^{ \a}} -\int_0^b H_t(r)  \f{  { 2\a r} }{ (1-r^2)^{ \a+ 1}} dr\right]\\
&\geq \esssup\limits_{t>0} c_{mn,\a} \left[H_t(b) \f{1}{ (1-b^2)^{ \a } } - \int_0^b M^m[\vec f\,](x)  \f{ { 2\a r}}{ (1-r^2)^{ \a+1}} v_{mn} r^{mn} dr\right]\\
&=c_{mn,\a} \left[ M^m[\vec f\,] (x) \f{v_{mn} b^{mn} }{ (1-b^2)^{ \a }}  - \int_0^b M^m[\vec f\,](x)
\f{  { 2\a r}}{ (1-r^2)^{ \a+ 1}  } v_{mn} r^{mn}dr\right]\\
&= c_{mn,\a} M^m[\vec f\,] (x)v_{mn}\left[ \f{b^{mn}}{ (1-b^2)^{ \a}}  - \int_0^b \f{ { 2\a r}}{ (1-r^2)^{ \a+1}} r^{mn}dr\right]\\
&=c_{mn,\a} M^m[\vec f\,] (x)v_{mn} \left[mn\int_0^b(1-r^2)^{-\a }r^{mn-1}dr\right],
\end{align*}
where all the previous steps make use of the assumption that $M^m[\vec f\,](x) <\infty$. 
Letting $b\to 1^{-}$ we obtain the first inequality in \eqref{lim}. 
 So we established both inequalities in  \eqref{lim} for $f_i\in L^1_{loc}(\R^n)$.

\medskip

Our next goal is to show that 
\begin{equation}\label{TBSmiddle}
\varliminf\limits_{ {\a\to 1^-}} S_\a^m[\vec f\,](x) \geq S^m [\vec f\,](x) , 
\end{equation}
 where $\varliminf\limits $ denotes the limit inferior. Let us fix $f_j $ in $L^1_{loc}(\R^n)$  and $x\in \R^n$. We define
 $$
 G_{\vec f\,} (t) =  \int_{\mathbb S^{mn-1}} | \!\otimes\! \vec f  \,| (\overline{x} -t  \vec \theta \,) d\sigma_{mn-1} (\vec \theta \, ). 
 $$
 We observed earlier that for any $L<\infty$ we have 
 \begin{equation}\label{eq:th1-tGL1loc}
 \int_0^L t^{mn-1}  G_{\vec f\,} (t) \, dt \le \prod_{i=1}^m  \int_{ (  |x|+
 { L} )\mathbb B^n } |f_i(y_i)|\, dy_i <\infty 
 \end{equation}
thus $t^{mn-1}  G_{\vec f\,} (t)$ lies in $L^1_{\textup{loc}} ([0,\infty))$ and 
$ G_{\vec f\,} (t) <\infty$ for almost all $t>0$. 
 {  We will show that for almost all $t>0$ we have 
\begin{equation}\label{TBSmiddle2}
\lim\limits_{ {\a\to 1^-}} \int_0^1G_{\vec f\,} (rt) \f{2r^{mn-1} (1-r^2)^{-\a} } { B(\f{mn}2,1-\a)} dr = G_{\vec f\,} ( t) .
\end{equation}
Once \eqref{TBSmiddle2} is shown, we deduce 
$$
\varliminf \limits_{ {\a\to 1^-}} { \esssup_{t'>0}} \int_0^1G_{\vec f\,} (rt') \f{2r^{mn-1} (1-r^2)^{-\a} } { B(\f{mn}2,1-\a)} dr \ge G_{\vec f\,} ( t)
$$
{ for almost all $t>0$, 
and taking the essential supremum  on the right over    $t>0$}, yields \eqref{TBSmiddle}.

For  smooth functions with compact support $\varphi_1, \dots, \varphi_m$ we have 
 \begin{equation}\label{TBSmiddle4}
\lim \limits_{ {\a\to 1^-}}  \int_0^1 \big| G_{\vec \varphi\,} (rt) -G_{\vec \varphi\,} ( t)\big| \f{2r^{mn-1} (1-r^2)^{-\a} } { B(\f{mn}2,1-\a)}dr =0  
\end{equation}
as
$$
\bigg| \prod_{i=1}^m \Big| \varphi_i(x-rt\theta_i) \Big|- \prod_{i=1}^m \Big|\varphi_i(x- t\theta_i) \Big|\bigg|  \le 
\bigg| \prod_{i=1}^m \varphi_i(x-rt\theta_i) - \prod_{i=1}^m \varphi_i(x- t\theta_i) \bigg| \le C\, t (1-r) 
$$
and this factor cancels the singularity of $(1-r^2)^{-\a}$ while $\lim\limits_{\a\to 1^-}B(\f{mn}2,1-\a)= +\infty$.  This implies \eqref{TBSmiddle2} with $\vec \varphi =(\varphi_1, \dots, 
\varphi_m) $ in place 
of $\vec f $, when each $\varphi_i$ is smooth and compactly supported.  
Next, we extend  \eqref{TBSmiddle2} to our fixed functions 
$f_j $ in $L^1_{loc}(\R^n)$.}
{

 Changing   variables $r'=  r t$
 we rewrite the integral in  \eqref{TBSmiddle2} as 
 \begin{eqnarray*}
 && \int_0^1 G_{\vec f\,}(rt) \f{2r^{mn-1} (1-r^2)^{-\a} } { B(\f{mn}2,1-\a)}  dr\\
 	&=&\f{2}{ B(\f{mn}2,1-\a)}  \f{1}{t^{mn}} \int_0^t G_{\vec f\,}(r') (r')^{mn-1} \Big( 1+\f{r'}{t}\Big)^{-\a} 
 	\Big( 1-\f{r'}{t}\Big)^{-\a} dr' \\
 	&=& \f{2}{ B(\f{mn}2,1-\a)t^{mn-\a}} \int_0^t G_{\vec f\,}(t-r) (t-r)^{mn-1} \Big( 2-\f{r}{t}\Big)^{-\a} 
 	r^{-\a} dr\\
 	&=:& \f{1}{ B(\f{mn}2,1-\a)t^{mn-\a}} T_\a\left((\cdot)^{mn-1}G_{\vec f\,}\right)(t). 
 \end{eqnarray*}
Note that
\[\f1{(1-\a) B(\f{mn}2,1-\a)t^{mn-\a}}\to \f1{t^{mn-1}}\]
as  $\a\to 1^-$. Thus, by setting $g(t)= t^{mn-1}  G_{\vec f\,}(t)$, we see that the   assertion   in  \eqref{TBSmiddle2}  can be equivalently reformulated as 
\begin{equation}\label{NEW3}
T_\a ( g  )   \to g \qquad\textup{a.e. on $(0,L)$ for every $L\in \mathbb Z^+$.}
\end{equation}
Let us fix such an $L\in \mathbb Z^+$ and note that $g=(\cdot)^{mn-1}G_{\vec f\,}\in L^1((0,L))$ by \eqref{eq:th1-tGL1loc}. In view of Proposition~\ref{p3} (for $m=1$) and the known convergence on a class of  smooth functions with compact support, it will suffice to know that the maximal operator 
\[ 
 T_* (h) (t)=\sup_{0<\a<1} T_\a (|h|)(t), 
 \]
maps $L^1((0,L)) $ to   $L^{1,\infty}(\mathbb R) $.

We introduce the   family of   kernels  
\begin{equation}\label{NEW1}
W_\a(r) = (1-\a) r^{-\a} \chi_{(0,\infty)}(r) , \qquad 
0<\a<1, 
\end{equation}
which form  an approximate identity and the operators
\(\tilde T_\a (g)(x) = \left(|g|\chi_{[0,\infty)}\ast W_\a \right).\)
 
Since $2^{-\a} \le (2-t/r)^{-\a} \le 1$  when $0\le r\le t$, 
$$
T_* (g)(t) \le \tilde T_* (g)(t): = \sup_{0<\a<1} \tilde T_\a (|g|)(t).
$$
Therefore it suffices to show that $\tilde T_*$  is pointwise controlled by the Hardy-Littlewood maximal operator $M$.  

For $0< t< L$ and $h\in L^1((0,L))$ we set $h=0$ outside $(0,L)$ and we write
 \begin{align*}
(1-\a)  \int_0^\infty |h(t-r)|\, \chi_{[0,\infty)}(t-r)  \, r^{-\a}dr & \le   (1-\a)2^\a 
\sum_{j=-\infty}^{    [\log_2 t]+1} 
\int_{2^{j-1}}^{2^j} |h(t-r)| 2^{-j\a} dr\\
& \le   (1-\a)2^\a  \sum_{j=-\infty}^{    [\log_2 t]+1}  \f{2^{j(1-\a)} }{2^j} 
\int_{0}^{2^j} |h(t-r)|   dr \\
& \le   (1-\a)2^\a  \Big( \int_{0}^{ 2t} y^{-\a} dy \Big) M(h) ( t)   \\
& \le   C \, t^{1-\a} 
M(h) ( t)  
\end{align*}
for some constant $C$ independent of $t>0$ and of $\a<1$. This gives 
   \begin{equation}\label{NEW5}
 T_*(h) (t) \le C  \, L\, M(h)(t) , \qquad t \in (0,L), 
 \end{equation}
and from this we derive the weak type $(1,1)$ property of $T_*$ on $(0,L)$ which, as observed earlier, implies  \eqref{TBSmiddle}.
}

\medskip

{  Finally, we prove \eqref{lim3}. To do this we fix $x$ in $\R^n$ and $f_i$ in $L^1_{loc}(\R^n)$ such that $S_{\a_0} ^m [\vec f\,](x) <\infty$ for some $\alpha_0\in (0,1)$. 
We first show that 
\begin{equation}\label{pp99}
\varlimsup\limits_{\a\to 0^+} \esssup_{t>0} K_{t, f, x} (\a) \le  
\esssup_{t>0}  \int_{\mathbb{B}^{mn}}  |\!\otimes\! \vec f  \,| (\overline{x}-t\vec y\, )\,   d\vec y .
\end{equation} 
where  we set
$$
K_{t, f, x}(\a) =\int_{\mathbb{B}^{mn}}   |\!\otimes\! \vec f  \, | (\overline{x}-t\vec y\, )  (1-|\vec y\,|^2)^{-\a }d\vec y,
\qquad t>0\, .
$$
Since we are taking the limit as $\a\to 0^+$ we may consider $\a<\f12 \a_0  =\a_1$. 
 By the fundamental    theorem of calculus, for $0< \a< \a_1$, we write
\begin{equation}\label{Kttt}
 K_{t, f, x}(\a) \le      K_{t, f, x}(0) +   \big|K_{t, f, x}(\a) - K_{t, f, x}(0)\big| 
  \le    K_{t,f,x}(0) +\a \sup_{\substack{ 0\le \beta\le \a_1  }} |K_{t, f, x}'(\beta)| , 
\end{equation}
where      $K_{t, f, x}'(\beta)$ denotes  the derivative of $K_{t, f, x}$ with respect to $\beta$. 
 We have
\[
|K_{t, f, x}'(\beta)|  = \int_{\mathbb{B}^{mn}} 
| \!\otimes\! \vec f  \,| (\overline{x}-t \vec y\, ) \Big(  \ln \frac{1}{ 1-|\vec y\, |^2}\Big)  
\f{d\vec y}{  (1-|\vec y\, |^2)^{ \beta } }    
\]
for  $0\le \beta\le \a_1<\a_0$  
and this is bounded by 
\[
c_{\a_0}  \int_{\mathbb{B}^{mn}} 
| \!\otimes\! \vec f  \,| (\overline{x}-t \vec y\, ) 
\f{d\vec y}{  (1-|\vec y\, |^2)^{ \a_0 } },   
\]
which is finite for almost every $t>0$ by assumption. 
Using this fact and 
taking the essential supremum in \eqref{Kttt} with respect to $t>0$,  we conclude   for $\a<\a_1$ that
\[
\esssup_{t>0}  K_{f , t, x}(\a)  \leq
\esssup_{t>0}   \int_{\mathbb{B}^{mn}}  |\!\otimes\! \vec f  \,| (\overline{x}-t\vec y\, )\,   d\vec y 
+ \a\, C_{\a_0,m,n }  S_{\a_0}^m[\vec f\,](x)  
\]
for some constant $C_{\a_0,m,n }<\infty$. 
In view of the fact that $S_{\a_0}^m[\vec f\,](x)<\infty$ we obtain
\[
 \varlimsup_{\a\to 0^+} \esssup_{t>0}  
 K_{f , t, x}(\a)    \leq \esssup_{t>0}  
 \int_{\mathbb{B}^{mn}}  |\!\otimes\! \vec f  \,|(\overline{x}-t\vec y\, )\,   d\vec y  .
 \]
This fact combined with  
\[
\lim_{\a\to 0^+}  \dfrac{2}{\omega_{mn-1}B(\f{mn}2,1-\a)}= \dfrac{1}{v_{mn }} 
 \]
implies that 
 \[
 \varlimsup_{\a\to 0^+} S_\a^m [\vec f\,] (x) \le  M^m [\vec f\,](x)  .
\]
 This yields   \eqref{lim3} since  we obviously  have  $ \varliminf\limits_{\a\to 0^+} S_\a^m [\vec f\,] (x)   \ge  M^m [\vec f\,](x)$ by \eqref{lim}.}
 \end{proof}

 \section*{The proof of Theorem~\ref{t2}}

\begin{proof}[Proof of Theorem~\ref{t2}]
 For any $0\leq \a<1$, we   prove that  the   estimate 
\begin{equation}\label{ptwise}   
S^m_\a (f_1,\dots,f_m)(x) \le  S_\a ( f_k)(x) \prod_{i\neq k} M (f_i) (x), 
\end{equation}
is valid for all   $f_i\in L^1_{loc}(\R^n)$ and all $x\in \R^n$,  
where $S_\a$ is defined in \eqref{Sa} and $M$ is the Hardy-Littlewood maximal operator on $\R^n$. 
For any fixed $t>0$, we set
\[
S^m_{\a,t} (f_1,\dots,f_m)(x) =    c_{mn,\a} \int_{\mathbb{B}^{mn}}\prod_{i=1}^m |f_i(x-ty^i)|(1-|y|^2)^{-\a } dy 
\]
where $ c_{mn,\a} = { 2} / (  {\omega_{mn-1} B( {mn}/2  ,1-\a)} ) $. 

For $y^i\in \R^n$ we set 
$$
y=(y^1,\dots,y^m)\qquad \textup{and} \qquad \hat{y}^k = (y^1,\dots,y^{k-1},y^{k+1}, \dots , y^m) . 
$$
Then for a fixed $k\in \{1,2,\dots ,m\}$ we have
\begin{align*} &c_{mn,\a} ^{-1}\,\, S^m_{\a,t} (f_1,\dots,f_m)(x) \\
&=    \int_{\mathbb{B}^{mn}}\prod_{i=1}^m |f_i(x-ty^i)|(1-|y|^2)^{-\a } dy\\
&=   \int_{\mathbb{B}^{(m-1)n}}   \int_{\sqrt{1\! - \!  |\hat{y}^k |^2} \mathbb{B}^{ n}}\prod_{i=1}^m |f_i(x-ty^i)   |  \,  (1-|\hat{y}^k |^2)^{-\a}  
  \bigg(1- \Big| \frac{  y^k} {   \sqrt{1\! - \!  |\hat{y}^k |^2}   }     \Big|^2\bigg)^{-\a } dy^k d\hat y^k\\
&=    \int_{\mathbb{B}^{(m-1)n}} \! \prod_{i\neq k} |f_i(x-ty^i)|  \! \int_{\mathbb{B}^n} 
\!\!\! |f_k(x-   t \sqrt{1-|\hat{y}^k |^2 }    u^k)|  \f{(1\! - \! |\hat y^k|^2)^{\f n2-\a}} {   (1\! - \! |u^k|^2)^{ \a }  }  du^k d\hat y^k\\
&\le     \int_{\mathbb{B}^{(m-1)n}}\prod_{i\neq k} |f_i(x-ty^i)| \esssup_{t>0} \int_{\mathbb{B}^n}|f_k(x-tu^k)| (1\! - \! |u^k|^2)^{-\a }du^k
 \f{d\hat y^k} {  (1\! - \! |\hat y^k|^2)^{\a -\f n2}  }  \\
& \leq c^{-1}_{n,\a} \,S_\a ( f_k) (x)\cdot  \sup_{t>0} \int_{\mathbb{B}^{(m-1)n}}\prod_{i\neq k} |f_i(x-ty^i)|  \f{d\hat y^k}{(1-|\hat y^k|^2)^{\a -\f n2}}, 
\end{align*}
with $c_{n,\a} = { 2} / (  {\omega_{n-1} B( {n}/2  ,1-\a)} ) $. 

Next, we use the following fact concerning multilinear approximate identities:  
Suppose that $\phi:\R^{\kappa n}\to \C$ has an integrable radially decreasing majorant $\Phi$, and let $\phi_t (\vec y\, ) = t^{-\kappa n}  \phi(\vec y/t)$. If $\ast $ denotes convolution on $\mathbb R^{\kappa n}$, then
the estimate 
\begin{equation}\label{maj} 
\sup_{t>0} | (\otimes  \vec{f}  \,) \ast \phi_t (\overline{ x}) | \leq \|\Phi\|_{L^1(\R^{\kappa n})} M^m[\vec{f}\,](x)
\end{equation}
is valid for all locally integrable functions $f_j$ on $\R^n$, $j=1,\dots,\kappa$. 
This follows by applying \cite[Corollary 2.12]{GClassical}  to the function    $(x_1,\dots, x_\kappa) \mapsto \otimes\vec{f} (x_1,\dots, x_\kappa) $ on  $\mathbb R^{\kappa n}$ and using that the $\kappa n$-dimensional 
Hardy-Littlewood maximal function of $\otimes\vec{f}$ at the point $
{ \overline{x} } = (x,\dots , x)\in (\R^n)^\kappa$ equals
$M^{ { \kappa }   } [\vec{f}\,] ({ \overline{x} }  )  .$

Returning to the previous calculation, for $\hat{y}^k\in\R^{(m-1)n}$ we consider the function $\phi(\hat{y}^k) = (1-|\hat{y}^k|^2)_{+}^{\f{n}{2} -\a}$. Using that $n\ge 2$ (hence $n/2-\a \ge 0$), we calculate that 
$$
\|\phi\|_{L^1(\R^{(m-1)n})} = \f{\omega_{(m-1)n-1}}2 \,B\left(\f{(m-1)n}2, \f{n}2 + 1 -\a\right).
$$
Using \eqref{maj} for $\kappa = m-1$, we can see that 
\begin{equation*} \sup_{t>0} \int_{\mathbb{B}^{(m-1)n}}\prod_{i\neq k} |f_i(x-ty^i)|  \f{d\hat y^k}{(1-|\hat y^k|^2)^{\a -\f n2}} \leq \|\phi\|_{L^1(\R^{\kappa n})} M^{m-1}[\hat{f}^k](x),
\end{equation*}
where $[\hat{f}^k] = (f_1,\dots,f_{k-1},f_{k+1},f_m).$ Using the well known fact that $\omega_{ n-1} = \f{2\pi^{n/2}}{\Gamma(n/2)}$ and the identity $B(a,b) = \f{\Gamma(a)\Gamma(b)}{\Gamma(a+b)}$, one can verify that 
\begin{equation*} c_{mn,\a}\cdot c^{-1}_{n,\a}\cdot\|\phi\|_{L^1} = 1.
\end{equation*}
Thus we conclude that 
\begin{equation*} S^m_{\a,t} (f_1,\dots,f_m)(x) \leq  S_\a (f_k)(x) M^{m-1}[\hat{f}^k](x).
\end{equation*}

Taking the essential supremum of 
$S^m_{\a,t} (f_1,\dots,f_m)(x)$ over $t>0$ yields   
\begin{equation}\label{BB} 
S^m_{\a} (f_1,\dots,f_m)(x) \leq  S_\a (f_k)(x) M^{m-1}[\hat{f}^k](x).
\end{equation}  
Since \eqref{BB} holds for $\a=0$, we have that
\begin{equation*} 
{ M^{m} } [\vec{f}\,] \leq M(f_1)(x) M^{m-1}[\hat{f}^{1}](x).
\end{equation*}  
Therefore, consecutive applications of 
{  this fact and} \eqref{BB} conclude  the proof of \eqref{ptwise}.

We now turn to the boundedness of $S_\a$ when $m=1$. It was shown in \cite{S1976} that $S_\a$ is bounded 
on $L^p$ for $ \frac{n}{n- \a}<p\le \infty$ when $n\ge 3$. We remark  that this result also holds when $n=2$.  
We now provide a sketch of a proof valid in all dimensions $n\ge 2$. To do this, for $f\in \mathcal S(\R^n)$, we express
$S_\a f$ as a maximal multiplier operator
$$
S_\a f(x) =  \frac{2\pi^{ \a}}{\omega_{ n-1}}\frac{\Gamma(1-\a)}{B( \f n2,1-\a)}  
\sup_{t>0} \bigg| \int_{\R^n} \widehat{f}(\xi) \f{J_{ \f n2-\a }  (2\pi t |\xi|) }{ |t\xi|^{\f n2-\a} } e^{2\pi i x\cdot \xi} d\xi\bigg|
$$
using the identity in \cite[Appendix B.5]{GClassical}. To derive this we   use the Bochner-Riesz multiplier
$(1-|x|^2)^{-\a}_{ +}$ with a negative exponent, viewed as a kernel. Then the Fourier transform expression for 
 $(1-|x|^2)^{z}_{ +}$ when $\textup{Re } z>0$   is also  valid for $\textup{Re } z>-1$ by analytic continuation. Notice that in this 
 range of $z$, the kernel remains locally integrable. Using properties of Bessel functions,  the multiplier
 $$
 m_\a(\xi) = \f{J_{ \f n2-\a }  (2\pi   |\xi|) }{ | \xi|^{\f n2-\a} }
 $$
 is a smooth function which satisfies for all multi-indices $\gamma$
 $$
| \partial_\xi^\gamma m_\a(\xi) | \le \f{ C_{n,\gamma} }{  | \xi|^{\f {n+1} 2-\a} }  
 $$
 and the exponent   $a = \f {n+1} 2-\a$ is strictly bigger than $\f12$ (since $n\ge 2$ and $\a<1$). Then 
 the hypotheses of \cite[Theorem B]{Rdf1986} apply and we obtain that $S_\a$ is bounded on $L^p(\R^n)$ 
 (when restricted to Schwartz functions) 
 for 
 $$
 p> \f{2n}{n+2a-1} = \f{n}{n-\a}.
 $$ 
(In  \cite[Theorem B]{Rdf1986} there is an   upper restriction  on $p$, but as $S_\a$ is bounded on $L^\infty$ this does not 
apply here.) Then $S_\a$  extends to general $f\in L^p(\R^n)$ for $p>\f{n}{n-\a}$ by density, and this
extension coincides with that given in Definition~\ref{Def1}. 

We now use \eqref{ptwise} to obtain that 
\begin{equation}\label{Int99}
\| S_\a^m (f_1,\dots , f_m) \|_{L^p(\R^n)} \le C\, \prod_{i=1}^m \|f_i\|_{L^{p_i}(\R^n) } 
\end{equation}
for all $f_i\in L^{p_i}$, when $1<p_i \le \infty$ for $i\neq k$ and $\f{n}{n-\a}<p_k\le \infty$. Here the constant $C=C(m,\a,p_1,\dots,p_m)$ doesn't depend on the dimension $n$, since $S_\alpha (f)(x) \le S_1 (f)(x)$ and $\|S_1 (f)(x)\|_{L^p(\R^n)} \le c\, \|f\|_{L^p(\R^n)}$ for a constant $c$ independent of $n$ (see~\cite{SS1983}).

To describe geometrically the points $(1/p_1, \dots , 1/p_m)$ for which we claim boundedness for $S^m_\a$, consider the cube $Q=[0,1]^m$ and let $V$ be the set of all of its vertices except for the vertex 
$(1,1,\dots , 1)$. Then $|V|=2^m-1$. We consider the intersection of $Q$ with the half-space $H$ of $\R^m$ described by 
$$ 
H=\big\{(t_1,\dots , t_m):\,\, t_1+\cdots +t_m \le  \tfrac{mn-\a}n \big\}.
$$ 
Then $Q\cap  H$ has $2^m-1+m$ vertices, namely the set $V$ union the $m$ points
$$ 
(1, \dots , 1 , \tfrac{mn-\a}n, 1, \dots , 1),
$$  
where $\tfrac{mn-\a}n$ ranges over  the $m$   slots.    We claim that  $S_\a^m$ satisfies strong $L^p$ bounds in the 
interior of $Q\cap H$. To see this, we interpolate between 
estimates at  the vertices of $Q\cap H$.  Precisely, the interpolation works as follows: Let 
  $W$ be the vertices of $Q\cap H$ that do not belong to $V$ and let 
    $W'$ be a finite union of open balls 
centered at the points of $W$ intersected with $Q\cap H$.
 We interpolate between  
points $P=(1/p_1, \dots , 1/p_m)$ in $V\cup W'$.    If $P\in V$, 
then we have an estimate $L^{p_1}\times \cdots \times L^{p_m}$ to $L^p$ for $S^m_\a$, 
as at least one coordinate $1/p_k$ is   $0$ (i.e., $p_k=\infty$), and we apply \eqref{Int99} for this $k$.  
Now if $P$ lies in $W'$, then there is a 
  $k\in \{1, \dots , m\}$ such that    $p_k>\f{n}{n-\a}$ and $p_i $  are near $1$ for all $i\neq k$.
Using estimate \eqref{Int99} again for this choice of $k$, we   obtain that  $S_\a^m$ maps $L^{p_1}\times \cdots \times L^{p_m}$ to $L^p$ 
 at this point $P$. Applying the $m$-linear version of the Marcinkiewicz interpolation theorem ~\cite{GLLZ2012}, 
 we deduce the boundedness of $S_\a^m$ in the interior of $Q\cap H$. 
Similar reasoning provides     weak type bounds on all the faces of  $Q\cap H$, except possibly on the $H$ face, on which 
 we don't know if there are any bounds at all.

 Finally we show the optimality of the range $p>\frac{n}{mn-\a  }$.  We consider the action of     $S^m_\a$ on characteristic 
functions; specifically, let $f_1=\dots=f_m= \chi_{\mathbb{B}^n} $. Since the characteristic functions belong in all $L^p$ spaces, in the definition of $S^m_\a[\vec{f }\,] $ we can replace the essential  supremum by the  supremum (see Corollary~\ref{C1}).  Therefore for $|x|$ sufficiently large it is enough to 
pick $t=\sqrt{m} \, |x|$ in order to write the estimate
\begin{align}
c^{-1}_{mn,\a}\, S^m_\a   (f_1,\dots,f_m)(x) &\ge   \int_{\mathbb{B}^{mn}} \prod_{i=1}^m |f_i(x-\sqrt{m} \, |x|\,  y_i)| \cdot (1-|\vec y\, |^2)^{-\a } d\vec y\notag \\
&\geq  \int_{|\vec y\,  -\f{   \overline{ x} }{ \sqrt{m} \,|x|} | \le   \f{ 1 }{ \sqrt{m} \,  |x|} }   (1-|\vec y\, |^2)^{-\a} d\vec y \notag\\
&\geq  2^{-\a} \int_{|\vec y\,  -\f{   \overline{ x} }{ \sqrt{m} \,|x|} | \le   \f{ 1 }{ \sqrt{m} \,  |x|} }   (1-|\vec y\, |)^{-\a} d\vec y\label{ninth}
\end{align}
since 
$$
\Big|\vec y\,  -\f{   \overline{ x} }{ \sqrt{m} \,|x|} \Big| \le   \f{ 1 }{ \sqrt{m}\, |x|} \implies \big|x-\sqrt{m}\, |x| \, y_j \big| \le 1
\qquad \textup{for all $j=1,\dots , m$.}
$$
The point $\overline { \theta_x}= \f{   \overline{ x} }{ \sqrt{m} |x|} $ lies in the   sphere $\mathbb S^{mn-1}$. A simple
geometric argument gives that the integral in \eqref{ninth} expressed in polar coordinates
$\vec y = r\vec \theta$ is at least 
$$
 \int_{ 1-\f{c}{|x|} }^1 (1-r)^{-\a }   r^{mn-1} \int_{|\vec\theta\, -\overline { \theta_x}|  \le \f{c}{|x|} }  
  d\sigma_{mn-1}( \vec \theta \,)\, \f{ dr }{ 2^{\a}}
  \ge  \f{2^{-\a}}{|x|^{1-\a}} \f{C(m,n)}{|x|^{mn-1}} = \f{2^{-\a}C(m,n)}{|x|^{mn-\a}}
 $$
 for some small constants $c,C$ (depending on $n$ and $m$). 
We conclude the proof by noting  that the function $|x|^{-mn+\a}\chi_{|x|\ge 100}$ 
does  not lie in $L^p(\R^n)$ for   $p\leq \frac{n}{mn-\a }$.
\end{proof}


We proved Theorem~\ref{t2} working directly with $L^{p_i}$ functions. Alternatively, we could have worked  with 
a dense family of $L^{p_i}$ and then extend to $L^{p_i}$ by density. There is no ambiguity in this extension, in view 
of the following proposition.

 \begin{prop}\label{Psub}
Let $0<p_1,\dots , p_m,p\le \infty$. 
Suppose that $T$ is a subadditive operator in each 
variable\footnote{this means $|T(\dots , f+g, \dots )|\le |T(\dots , f , \dots) |+T(\dots ,  g, \dots) |$ for all $f,g$} that satisfies the estimate 
\begin{equation}\label{KK}
\| T(f_1,\dots , f_m) \|_{L^p} \le K \|f_1\|_{L^{p_1}} \cdots \|f_m\|_{L^{p_m}}
\end{equation}
for all functions $f_j$ in a dense subspace of $L^{p_j}$. Then $T$ admits a unique bounded subadditive extension from  
$L^{p_1}\times \cdots \times L^{p_m}$ to $L^p$ with the same bound. 
\end{prop}

\begin{proof}
For any $j\in \{1,\dots , m\}$, 
given $f_j\in L^{p_j}$ pick sequences $a_j^k, b_j^l$, $k,l=1,2,3,\dots$ in the given dense subspace of $L^{p_j}$ 
which converge to $f_j$ in $L^{p_j}$. 
Using the idea proving \eqref{ab} 
 and the subadditivity of $T$ in each variable we obtain:
$$
\big| T(a_1^k,a_2^k,\dots ,a_m^k) - T(b_1^l , b_2^l, \dots , b_m^l) \big| \le  
$$
$$
\sum_{i=1}^m\Big[ \big| T( b_1^l ,\dots ,b_{i-1}^l , a_i^k-b_i^l,  a_{i+1}^k,\dots, a_m^k )\big| +
  \big| T( a_1^k ,\dots ,a_{i-1}^k , b_i^k-a_i^l,  b_{i+1}^l,\dots, b_m^l )\big| \Big].
$$
Applying the $L^p$ (quasi norm) and hypothesis \eqref{KK} we deduce
\begin{align}\begin{split}\label{876}
& \big\| T(a_1^k,a_2^k,\dots ,a_m^k) - T(b_1^l , b_2^l, \dots , b_m^l) \big\|_{L^p} \\
& \qquad\qquad \le C_{p } K\sum_{i=1}^m \| a_i^k-b_i^l\|_{L^{p_i}} \prod_{j\neq i} \big[ \| a_j^k\|_{L^{p_j}} + \| b_j^l\|_{L^{p_j}}\big] .
\end{split}\end{align}
Taking $b_j^l = a_j^l$ in \eqref{876} we conclude  that the sequence 
$ 
\big\{ T(a_1^k,a_2^k,\dots ,a_m^k)\big\}_{k=1}^\infty
$ 
is Cauchy in $L^p$ and thus it has a limit $\overline{T}(f_1,\dots , f_m)$. This limit does not depend on the choice
of the sequences $ a_j^k $ converging to $ f_j$, as we can choose $l=k$ in \eqref{876} and let $k\to \infty$. 
Thus $T$ has a unique extension $\overline{T}$. This
extension is also bounded with the same bound and is subbadditive by density. 
\end{proof}

\section*{The proofs of Corollaries~\ref{C1} and ~\ref{C2}}  

Next we discuss the proof of Corollary~\ref{C1}. The case $m=1$ of this result is contained in \cite[Chapter XI Section 3.5]{SHarmonic}.
\begin{proof}  It suffices to   prove the assertion  for almost all $x $ in a ball $N\, \mathbb{B}^n $, as  $\R^n$ is 
a countable union of $N\, \mathbb{B}^n $ over $N=1,2,\dots$. Let us fix  such a ball  $N\, \mathbb{B}^n $. It will suffice to prove the 
continuity of $t\mapsto S_{\a, t}(f_1, \dots , f_m)(x)$ 
 on $(0,R)$ for every $R>0$. Fix such an $R>0$ as well. Then we may replace each $f_i$ by 
$g_i=f_i\chi_{(N+R) \,\mathbb{B}^n }$ as   $S_{\a, t}(f_1, \dots , f_m)(x) =S_{\a, t}(g_1, \dots , g_m)(x)$ when $x\in N \,\mathbb{B}^n $
and $0<t<R$. As $g_i$ have compact support and lie in $L^{p_i}$, they also lie in $L^{q_i}$, where $q_i<p_i$ are chosen so that $\f{1}{q}=\sum_{i=1}^m \f{1}{q_i} <\f{mn-\a}{n}$.  (The purpose of introducing $q_i<p_i$
 was   to replace all 
infinite indices $p_i $ by   finite ones, as there is no good dense subspace of $L^\infty$.) 

 We pick   sequences $\varphi_j^k$ of smooth compactly supported functions with   $\varphi_j^k
\to    g_j $  in $L^{q_j}(\R^n)$ (since $q_j<\infty$) and consider the sequence
$$
\esssup_{t>0} S^m_{\a , t} (g_1-\varphi^k_1, \dots , g_m-\varphi^k_m), \qquad { k} =1,2,3,\dots .
$$
By     \eqref{B} if $\a<1$ (or by \cite{D2019} if $\a=1$) this sequence converges to zero in $L^q(\R^n)$, thus there is a subsequence that converges to zero a.e. 
This implies that there is a subset $E$ of $\R^n$ of measure zero such that for   all $x\in\R^n\setminus E$ we have
\[
\lim_{k\to\infty}  \big\| S^m_{\a,t} [\vec{g}\, ](x) - S^m_{\a,t} [\vec{\varphi^k}] (x) \big\|_{L^\infty( (0,\infty),dt)} = 0 , 
\]
i.e., $S^m_{\a,t} [\vec{\varphi^k}] (x) \to S^m_{\a,t} [\vec{f}\, ](x)$ uniformly in $t> 0$.
Since $S^m_{\a,t} [\vec{\varphi^k}](x)$ is continuous in $t$, we conclude that $S^m_{\a,t} [\vec{g}\,](x)$ is also continuous in $t$, for almost every $x\in\R^n$.
\end{proof}

To prove Corollary~\ref{C2}, we will need a proposition analogous to~\cite[Theorem 2.1.14]{GClassical}. Let $(X,\mu)$ and $(Y,\nu)$ be $\sigma$ finite measure spaces and let $0<p_j\le \infty$, $j=1,\dots,m$, and $0<q<\infty$. Let $D_j$ be a dense subspace of $L^{p_j}(X,\mu)$. Suppose   that for all $t>0$, $T_t$ is  an $m$-linear operator  defined on $L^{p_1}(X,\mu)\times\dots\times L^{p_m}(X,\mu) $ with values in the space of measurable functions defined a.e. on $Y$.  Assume that for all $f_j\in L^{p_j}$, the   function 
$$
y\mapsto T_{\ast} (f_1,\dots,f_m)(y) = \sup_{t>0}|T_t (f_1,\dots,f_m)(y)| 
$$
is $\nu$-measurable on $Y$.

\begin{prop}\label{p3} Let $0<p_i \le\infty$, $1\le i\le m$, $0<q<\infty$ and $T_t$ and $T_\ast$ as in the previous discussion.  Suppose that there is a constant $B$ such that 
\begin{equation}\label{100}
\|T_{\ast} (f_1,\dots,f_m)\|_{L^{q,\infty}} \leq B \prod_{j=1}^m \|f_j\|_{L^{p_j}}
 \end{equation}
 for all $f_j\in L^{p_j}(X,\mu)$. Also suppose  that for all $\vp_j\in D_j$
\begin{equation}\label{tto1}
 \lim_{t\to 0} T_{t} (\vp_1,\dots, \vp_m) = T(\vp_1,\dots,\vp_m)
 \end{equation}
exists and is finite $\nu$-a.e. Then for all functions $f_j\in L^{p_j}(X,\mu) $ the limit in \eqref{tto1} exists and is finite $\nu$-a.e., and defines an $m$-linear operator   which uniquely extends $T$ defined on $D_1\times \cdots \times D_m$ and
which is bounded from $L^{p_1}\times \cdots \times L^{p_m}$ to $L^{q,\infty}(X)$. 
\end{prop}
 
 \begin{proof} Given $f_j\in L^{p_j}$ we define the oscillation of $\vec{f}$ for $y\in Y$ by setting
\begin{equation*}
O_{\vec{f}}\, (y) = 
\limsup_{\varepsilon \to 0} \limsup_{\theta \to 0} \big|T_\varepsilon [\vec{f}\,](y) -  T_\theta[\vec{f}\,](y) \big|.
\end{equation*}
We will show that for all $f_j\in L^{p_j}$ and all $\delta>0$,
\begin{equation}\label{o0}
\nu\big(\{y\in Y\,:\,O_{\vec{f}}\, (y) > \delta\}\big) = 0.
\end{equation}
Once  \eqref{o0} is established, we obtain that $O_{\vec{f}}\, (y) = 0$ for $\nu$-almost all $y$, which implies that $T_t[\vec{f}\,](y)$ is Cauchy for $\nu$-almost all $y$, and it therefore converges $\nu$-a.e.~to some $T[\vec{f}\,](y)$ as $t\to 0$. The operator $T$ defined this way on $L^{p_1}(X)\times\dots\times L^{p_m}(X)$ is linear and extends $T$ given in \eqref{tto1} defined on $D_1\times \cdots \times D_m$.

To approximate $O_{\vec{f}}\, (y)$ we use density. Given $0<\eta<1$, we find $\vp_j\in D_j$ such that $\|f_j- \vp_j\|_{L^{p_j}}<\eta$, $j=1,\dots,m$. Without a loss of generality, we also assume that $\|\vp_i\|_{L^{p_i}} \leq 2\|f_i\|_{L^{p_i}}$. Since $T_t[\vec{\vp}\,] \to T[\vec{\vp}\,]$ $\nu$-a.e., it follows that $O_{\vec{\vp}}\, = 0$ $\nu$-a.e. Using \eqref{ab}, we write 
$$
T_t [\vec{f}\,]  - T_{t}[\vec{\vp}\,] = \sum_{i=1}^m T_t(\vp_1,\dots,\vp_{i-1}, f_i-\vp_i, f_{i+1},\dots,f_m)  
$$
and from this  we obtain  
$$
O_{\vec{f}}   \le O_{\vec{\vp}}   + \sum_{i=1}^m O_{(\vp_1,\dots,\vp_{i-1}, f_i-\vp_i ,f_{i+1},\dots,f_m)} \qquad \nu\text{-a.e.}
$$
Now, for any $\delta>0$ we have 
\begin{align*}
&\hspace{-.1in}\nu\big(\{y\in Y\,:\,O_{\vec{f}}\, (y) > \delta\}\big)\\ &\le \nu\Big(\Big\{y\in Y\,:\,\sum_{i=1}^m O_{(\vp_1,\dots,\vp_{i-1}, f_i-\vp_i ,f_{i+1},\dots,f_m)} > \delta\Big\}\Big) \\
&\le \nu\Big(\Big\{y\in Y\,:\,\sum_{i=1}^m 2T_\ast (\vp_1,\dots,\vp_{i-1}, f_i-\vp_i ,f_{i+1},\dots,f_m) > \delta\Big\}\Big)\\
&\le \sum_{i=1}^m \nu\Big(\Big\{y\in Y\,:\, 2T_\ast (\vp_1,\dots,\vp_{i-1}, f_i-\vp_i ,f_{i+1},\dots,f_m) > \f\delta m \Big\}\Big)\\
&\leq  \sum_{i=1}^m \Big[\Big(2B \f{m}{\delta}\Big) \|\vp_1\|_{L^{p_1}}\cdots \|\vp_{i-1}\|_{L^{p_{i-1}}} \|f_i-\vp_i\|_{L^{p_i}} \|f_{i+1}\|_{L^{p_{i+1}}}\cdots \|f_{m}\|_{L^{p_{m}}}\Big]^q\\
&\leq \Big(2^m B \f m \delta\Big)^q \eta^q \sum_{i=1}^m \big( \prod_{j\neq i} \|f_j\|^q_{L^{p_j}} \big).
\end{align*}
Letting $\eta\to 0$, we deduce \eqref{o0}. We conclude that $T_t[\vec{f}\, ]$ is a Cauchy sequence and hence it converges $\nu$-a.e.~to some $T[\vec{f}\,]$ which satisfies the claimed assertions. 
\end{proof}  

 We   now     prove Corollary~\ref{C2}
\begin{proof} 
It suffices to   prove the assertion  for almost all $x $ in a ball $N\, \mathbb{B}^n $, as  $\R^n$ is 
a countable union of balls. Let us fix   a ball $N\, \mathbb{B}^n $.   Then we 
  replace the given $f_i$ in $L^{p_i}_{loc}$ by 
$g_i=f_i\chi_{(N+1)\, \mathbb{B}^n }$ since   $S_{\a, t}(f_1, \dots , f_m)(x) =S_{\a, t}(g_1, \dots , g_m)(x)$ when $x\in N\, \mathbb{B}^n $
and $0<t<1$. As $g_i$ have compact support and lie in $L^{p_i}$, they also lie in $L^{q_i}$, where $q_i<p_i$ are chosen so that $\f{1}{q}=\sum_{i=1}^m \f{1}{q_i} <\f{mn-\a}{n}$. As $q_i<\infty$, the space of 
smooth functions with compact support is a dense 
subspace of $L^{q_i}$. Now \eqref{99} is easily shown to hold for smooth functions with 
compact support $f_i$, when $0\le \a\le 1$, thus \eqref{tto1} holds with 
$T_t=S_{\a ,t}$. Moreover \eqref{100} holds by Theorem~\ref{t2} if $\a<1$ or by \cite{D2019} if $\a=1$. 
By Proposition~\ref{p3}, for $t<1$, we obtain that  for almost all $x\in N\, \mathbb{B}^n $  we have   
\begin{align*} 
 \lim_{t\to 0}  S_{\a,t}^m (f_1,\dots, f_m) (x)   
=  \lim_{t\to 0}S_{\a,t}^m (g_1,\dots, g_m) (x)  
=   \prod_{j=1}^m g_j(x) 
=   \prod_{j=1}^m f_j(x),  
\end{align*}
thus   \eqref{99} holds for all   $g_i$ in $L^{q_i}$, in particular   for our given $f_i$ in $L^{p_i}_{loc}$.
\end{proof}

\bibliographystyle{amsplain}

\Addresses
\end{document}